\documentclass[11pt]{article}
\usepackage{amsfonts,amsmath,amsthm,amssymb,enumerate,color}
\usepackage[makeroom]{cancel}
\usepackage{authblk}

\usepackage[shortlabels]{enumitem}
\setlist[enumerate,1]{wide, labelindent=0pt,label={\upshape(\roman*)}}

\numberwithin{equation}{section}
\newcommand{\beq}{\begin{equation}}
\newcommand{\eeq}{\end{equation}}
\newcommand{\bea}{\begin{eqnarray}}
\newcommand{\eea}{\end{eqnarray}}
\newcommand{\beas}{\begin{eqnarray*}}
\newcommand{\eeas}{\end{eqnarray*}}

\newtheorem{theorem}{Theorem}[section]

\newtheorem{proposition}[theorem]{Proposition}

\newtheorem{corollary}[theorem]{Corollary}
\newtheorem{lemma}[theorem]{Lemma}
\newtheorem{remark}[theorem]{Remark}
\newtheorem{example}[theorem]{Example}
\newtheorem{examples}[theorem]{Examples}
\newtheorem{foo}[theorem]{Remarks}

\newcommand{\ang}[1]{\left<#1\right>}  

\newcommand{\bM}{\mathbb M}

\newcommand{\Riem}{\mathrm{Riem}}
\newcommand{\Ho}{\mathcal H}
\newcommand{\V}{\mathcal V}
\newcommand{\hor}{\mathcal H}
\newcommand{\ver}{\mathcal V}

\newcommand{\M}{\mathbb M}
\newcommand{\Tor}{\mathrm{Tor}}

\newcommand{\be}{\beta}

\newcommand{\ch}{\mathcal H}

\newcommand{\Dh}{{\Delta}_{\mathcal{H}}}

\title{Sub-Laplacian comparison theorems on  Riemannian foliations with minimal leaves and applications}

\author{Fabrice Baudoin
}

\affil{Department of Mathematics, Aarhus University\par
   \texttt{fbaudoin@math.au.dk}\vspace{1em}}
\begin{document}

\maketitle

\newcommand\blfootnote[1]{%
  \begingroup
  \renewcommand\thefootnote{}\footnote{#1}%
  \addtocounter{footnote}{-1}%
  \endgroup
}
\blfootnote{Research supported in part by a Villum Investigator grant and the Advanced Grant \emph{RanGe} from the European Research Council.}

\begin{abstract}
We prove comparison theorems for the horizontal Laplacian of the Riemannian distance in the context of Riemannian foliations with minimal leaves. This general framework generalizes  previous works and allow us to consider the sub-Laplacian of Carnot groups of arbitrary steps. The comparison theorems yield a Bonnet-Myers type theorem, stochastic completeness and Lipschitz regularization property for the sub-Riemannian semigroup.
\end{abstract}

\

\tableofcontents

\section{Introduction}

Comparison theorems for the Laplacian of the distance function are among the
central tools of Riemannian geometry. They provide a direct link between lower
bounds on curvature and global geometric or analytic consequences, such as
diameter bounds, compactness criteria, or functional inequalities. A classical
example is the Bonnet--Myers theorem, which follows from Laplacian comparison
under a positive lower Ricci curvature bound. 

In this work, we extend this circle of ideas to the setting of Riemannian
foliations with minimal leaves. More precisely, we prove comparison theorems
for the upper bound of the \emph{horizontal Laplacian} of the Riemannian
distance function. On such foliations, the horizontal Laplacian is a subelliptic operator that can be
characterized as the partial trace of the Hessian in horizontal directions for
the Levi-Civita connection. Under the classical assumption of a lower bound on
sectional curvature, horizontal Laplacian comparison follows rather directly
from the Hessian comparison theorem. Our goal, however, is to go beyond this
observation: by exploiting the geometry of the foliation itself, we establish
comparison theorems under the assumption of a lower bound on a
\emph{Ricci-like tensor} $\mathfrak R$ that naturally arises from the horizontal distribution of the foliation. Since
the Levi-Civita connection does not, in general, preserve horizontality, the
tensor $\mathfrak R$ is not simply the horizontal trace of the Riemann curvature tensor,
but instead incorporates  curvature and torsion contributions from the geometry of the
horizontal distribution. 

Our approach therefore generalizes, and significantly extends, the work of
\cite{radial} and \cite{BauGr}, where the case of totally geodesic foliations was considered. In particular,
our framework applies to all Carnot groups of arbitrary step.

\subsection*{Main results}

The main results of this paper may be summarized as follows:

\begin{itemize}
    \item \textbf{Horizontal Laplacian comparison theorems.}
    We establish Riemannian model upper bounds for the horizontal Laplacian of the
    Riemannian distance function under the curvature assumption
    \[
        \mathfrak R(U,U) \geq K \, |U|^2, \qquad U \in TM,
    \]
    where $\mathfrak R$ is the horizontal Ricci-type tensor associated with the foliation. The
    resulting inequalities take the form of  non-sharp but useful comparison formulas analogous
    to their Riemannian counterparts, involving the cotangent, linear, or
    hyperbolic cotangent functions depending on the sign of $K$.

    \item \textbf{A Bonnet--Myers type theorem.}
    As a direct application of the comparison results, we prove that if
    $\mathfrak R \geq K > 0$, then the ambient manifold is compact, with an explicit
    upper bound on its Riemannian diameter of the form
    \[
        \mathrm{diam}(M) \leq \pi \sqrt{\frac{n}{K}},
    \]
   obtaining a Bonnet--Myers theorem in the sub-Riemannian
    framework of foliations with minimal leaves.

    \item \textbf{Stochastic completeness and heat kernel estimates.}
    We study the horizontal Brownian motion associated with the generator
    $\Delta_H$. Using the radial decomposition and comparison arguments, we
    show that the corresponding heat semigroup is stochastically complete.
    Moreover, we obtain  exit time estimates which in turn yield
    on-diagonal lower bounds for the heat kernel.

    \item \textbf{Lipschitz regularization and coupling techniques.}
    By developing couplings of horizontal Brownian motions via parallel
    transport, we prove that the horizontal heat semigroup enjoys Lipschitz
    regularization properties: for any bounded Lipschitz function $f$,
    \[
        \mathrm{Lip}(P_t f) \leq e^{-Kt} \, \mathrm{Lip}(f).
    \]
    This type of estimate is of independent interest in analysis and
    probability, as it plays a key role in the theory of generalized curvature
    dimension inequalities in \cite{BG170}. In the case of totally geodesic foliations, our
    method yields even sharper \emph{gradient bounds} of the form
    \[
        |\nabla P_t f| \leq e^{-Kt} \, P_t |\nabla f|,
    \]
    thereby recovering and extending known results from \cite{MR3601645}.

    \item \textbf{Applications to Carnot groups.}
    Since any Carnot group fits into our framework as a Riemannian foliation
    with minimal leaves, our results provide a new and systematic approach to
    sub-Laplacian comparison theorems for the Riemannian distance in Carnot groups of arbitrary step. In
    step two Carnot groups, such as the Heisenberg group, we also recover known
    mixed horizontal/vertical gradient bounds.
\end{itemize}

\subsection*{Outlook}

The comparison theorems established in this paper open the way to a broader
theory of curvature-dimension and functional inequalities in sub-Riemannian
geometries, beyond the cases treated in \cite{BBG14,BG170}. In the setting
of totally geodesic Riemannian foliations, a quick computation shows that the family of
tensors $\mathfrak{R}_\varepsilon$ arising from the canonical variation of the
metric give the tensors governing the generalized curvature-dimension
inequality of \cite{BBG14}. Our results here suggest that, in the more general
framework of foliations with minimal leaves, an analogous theory of
curvature-dimension inequalities should emerge, again based on the
one-parameter family of tensors $\mathfrak{R}_\varepsilon$. Developing this
theory is a natural continuation of the present work and will be addressed in
future research.

\medskip

Finally, it is worth noting that the comparison theorems proved here are not yet
fully sub-Riemannian: they provide estimates for the horizontal Laplacian of
the \emph{Riemannian} distance, rather than the sub-Riemannian distance.
Obtaining analogous bounds for the sub-Riemannian distance is substantially more
challenging and typically requires much stronger structural assumptions, see
\cite{baudoin2019comparison}. Nevertheless, such sub-Riemannian comparison
theorems have stronger implications, for instance leading to the measure
contraction property, which cannot be deduced from horizontal Laplacian
estimates of the Riemannian distance.

\

\textbf{Notations:}

\begin{itemize}
\item If $M$ is a manifold, $TM$ is the tangent bundle.
\item $\mathcal{L}$ is the Lie derivative
 \item If $\mathcal{W}$ is a vector bundle over $M$, $\mathfrak{X}(\mathcal{W})$ is the set of smooth sections of that bundle.
 \item If $g$ is a Riemannian metric we denote $\left\langle u, v \right\rangle=g(u,v)$, $|u|^2=g(u,u)$.
\end{itemize}

\section{Horizontal Laplacian comparison theorem} \label{sec:RFoliation}

\subsection{Framework and standing assumptions} \label{ssec:foliation}

Throughout the paper, we consider a smooth connected $n+m$ dimensional manifold $M$ which is equipped with a Riemannian foliation $\mathcal{F}$ with a bundle like metric $g$ and $m$ dimensional leaves.  We  assume that the metric $g$ is complete and that the leaves are closed.

\

For $x \in M$, $\mathcal{F}_x$ denotes the leaf going through $x$. The sub-bundle $\mathcal{V}$ of the tangent bundle $TM$ formed by vectors tangent to the leaves is referred  to as the set of \emph{vertical directions}. The sub-bundle $\mathcal{H}$ which is normal to $\mathcal{V}$ is referred to as the set of \emph{horizontal directions}.  Any vector $u \in T_xM$ can  be decomposed as 
\[
u=u_\mathcal{H} +u_\mathcal{V}
\]
where $u_\mathcal{H}$ (resp. $u_\mathcal{V}$) denotes the orthogonal projection of $u$ on $\mathcal{H}_x$  (resp. $\mathcal{V}_x$).  Throughout the paper we assume that the horizontal sub-bundle $\Ho$ is bracket generating. Recall that saying that the metric is bundle like means that:
\begin{align}\label{bundle like condition}
 (\mathcal{L}_U g)(X, X)=0, \qquad \text{for any $X \in \mathfrak{X}(\Ho)$, $U \in \mathfrak{X}(\V)$.}
 \end{align}
Geometrically,  \eqref{bundle like condition} means that the foliation can locally be described by Riemannian submersions. We refer for instance  to the classical reference  \cite{Tondeur} or the more recent monograph \cite{Gromoll} for an overview of the theory of Riemannian foliations.

In this paper we will be interested in foliations with minimal leaves, i.e. we will assume all the leaves to be minimal submanifolds of $M$. Recall that a submanifold is called minimal if its mean curvature is zero. In terms of the Levi-Civita connection $D$ saying that the leaves are minimal is equivalent to 
\[
\sum_{i=1}^m D_{Z_i}Z_i \in \mathfrak{X}(\V)
\]
where the $Z_i$'s form an arbitrary (local) orthonormal frame of $\V$. There are many examples of such foliations. For instance, a large class of such foliations arise from the situation where the leaves are totally geodesic which corresponds to the case where $U,V \in \mathfrak{X}(\V)$ implies $D_UV \in \mathfrak{X}(\V)$. For rank one foliations (i.e. $m=1$), minimality and total geodesicity are the same, but of course not all minimal foliations are totally geodesic. For instance, consider the $2n+1$-dimensional sphere $\mathbf{S}^{2n+1}$ seen as a subspace of $\mathbb{C}^{n+1}$ and the manifold
\[
M=\mathbf{S}^{2n+1} \setminus \left\{ (z_1,\cdots,z_{n+1}) \in \mathbf{S}^{2n+1} : \exists i, z_i=0  \right\}
\]
The group $\mathbf{U}(1)^{n+1}$ acts on $M$ by:
\[
(e^{i\theta_1},\cdots, e^{i\theta_{n+1}})\cdot (z_1,\cdots,z_{n+1})=(e^{i\theta_1} z_1,\cdots,e^{i\theta_{n+1}} z_{n+1}).
\]
The orbits of that action yield a Riemannian foliation on $M$ whose leaves are tori. From O'Neilll's equation (see \cite[Corollary 1.5.1]{Gromoll}) the foliation can not be totally geodesic because the leaves are flat whereas the ambient space $M$ has positive constant sectional curvature. However, it is easily seen by symmetry that the orbits are minimal submanifolds. 

The main example of foliations with minimal leaves we consider throughout the paper is the case of Carnot groups.  A Carnot group is a connected, simply connected nilpotent Lie group $G$ whose Lie algebra  $\mathfrak{g}$ admits a stratification
\begin{align}\label{stratification}
\mathfrak{g} = V_{1} \oplus V_{2} \oplus \cdots \oplus V_{s},
\end{align}
with the properties:
\begin{enumerate}
    \item $[V_{1}, V_{j}] = V_{j+1}$ for every $1 \leq j < s$,
    \item $[V_{1}, V_{s}] = \{0\}$.
\end{enumerate}
The integer $s$ is called the step of the Carnot group, and $V_{1}$ is referred to as the horizontal layer (or first layer). Consider on $\mathfrak{g}$ an arbitrary inner product that makes the decomposition \eqref{stratification} orthogonal, i.e. for $i \neq j$, $V_i \perp V_j$. This inner product uniquely defines a left-invariant Riemannian metric on $G$. We can orthogonally decompose the tangent bundle $TG$ as
\[
TG=\mathcal{H}\oplus \mathcal{V}
\]
where $\Ho$ is the left invariant sub-bundle which gives $V_1$ at the identity and $\V$ is the left invariant sub-bundle which gives $\oplus_{i \ge 2} V_i$ at the identity. Since $[\mathcal{V},\mathcal{V}] \subset \mathcal{V}$, $\V$ is the vertical bundle of a foliation $\mathcal{F}$ on $G$. Note that $\Ho$ is bracket generating.

\begin{proposition}\label{Example Carnot}
Let $G$ be a Carnot group equipped with the foliation $\mathcal{F}$ and the metric $g$. Then, $\mathcal F$ is a Riemannian foliation with bundle-like metric $g$ and the leaves of the foliations are minimal submanifolds of $G$. 
\end{proposition}

\begin{proof}
 For $X \in \mathfrak{X}(\Ho)$ and $U \in \mathfrak{X}(\V)$ which are left-invariant one has
\[
(\mathcal{L}_U g)(X, X)=-2 \left\langle [U,X], X \right\rangle =0
\]
because at the identity one has $X \in V_1$ and $[U,X] \in \oplus_{i \ge 2} V_i$. Therefore, the metric $g$ is bundle-like and the foliation $\mathcal F$ Riemannian. Then, consider a left-invariant vertical orthonormal frame $Z_i$ adapted to the grading in the sense that for every $i$ there exists $n(i)$ such that $Z_i \in V_{n(i)}$ at  identity. From Koszul's formula, for any left invariant horizontal vector $X$  one has
\[
\left\langle D_{Z_i} Z_i ,X \right\rangle=-\left\langle [Z_i,X],Z_i \right\rangle=0
\]
because at identity $[Z_i,X] \in V_{n(i)+1}$. Therefore $\sum_{i=1}^n D_{Z_i} Z_i \in \mathfrak{X}(\V)$.
\end{proof}

Note that it follows from similar arguments that the foliation on Carnot group is totally geodesic if and only if the group has step 2.

\subsection{Horizontal Laplacian and fundamental connection}

We  come back to the general framework described in the previous section. The Riemannian gradient will be denoted $\nabla$ and the horizontal gradient $\nabla_\Ho$: it is simply defined as the projection of $\nabla$ onto $\mathcal{H}$. The horizontal Laplacian $\Delta_\Ho$ is the generator of the symmetric closable bilinear form:
\[
\mathcal{E}_{\mathcal{H}} (f,g) 
=-\int_\bM \langle \nabla_\mathcal{H} f , \nabla_\mathcal{H} g \rangle_{\mathcal{H}} \,d\mu, 
\quad f,g \in C_0^\infty(M),
\]
where $\mu$ denotes the Riemannian volume measure on $M$ and  $C_0^\infty(M)$ the space of smooth and compactly supported functions on $M$.

The hypothesis that $\mathcal{H}$ is bracket generating implies that the horizontal Laplacian $\Delta_{\mathcal{H}}$ is locally subelliptic and the completeness assumption on $g$ implies that $\Delta_{\mathcal{H}}$ is furthermore essentially self-adjoint on $C_0^\infty(M)$ (see for instance~\cite{BaudoinEMS2014}).

If $X_1,\cdots,X_n$ is a local orthonormal frame of horizontal vector fields, then we locally have
\begin{align}\label{def: horizontal laplace}
\Delta_{\mathcal{H}}=\sum_{i=1}^n X_i^2 -\sum_{i=1}^n (D_{X_i}X_i)_{\mathcal{H}}.
\end{align}

The Levi-Civita connection $D$ associated with the Riemannian metric $g$ is in general not well adapted to study foliations since the horizontal and vertical bundle might not be $D$-parallel.  There is a more natural connection $\nabla$   that respects the foliation structure, see    \cite{Hladky}, \cite{baudoin2019comparison} and \cite{AGAG}.  

Define a $(2,1)$ tensor $ C$ by the formula: For $X,Y,Z \in \mathfrak{X}(TM)$
\begin{equation}
\left\langle C_X Y, Z\right\rangle = \frac{1}{2}(\mathcal{L}_{X_\hor}g)(Y_\ver,Z_\ver).
\end{equation}
The connection $\nabla$ is then defined in terms of the Levi-Civita one $D$ by
\[
\nabla_X Y =
\begin{cases}
 ( D_X Y)_{\mathcal{H}} , &X,Y \in \mathfrak{X}(\mathcal{H}), \\
  [X,Y]_{\mathcal{H}},  &X \in \mathfrak{X}(\mathcal{V}),\ Y \in \mathfrak{X}(\mathcal{H}), \\
  [X,Y]_{\mathcal{V}}+C_X Y,  &X \in \mathfrak{X}(\mathcal{H}),\ Y \in \mathfrak{X}(\mathcal{V}), \\
 (D_X Y)_{\mathcal{V}} , &X,Y \in \mathfrak{X}(\mathcal{V}).
\end{cases}
\]
It satisfies the following characteristic (and defining) properties:
\begin{enumerate}
\item $\Ho$ and $\V$ are $\nabla$-parallel: For every $X \in \mathfrak{X}(\mathcal{H})$, $Z \in \mathfrak{X}(TM)$ and $U \in \mathfrak{X}(\mathcal{V})$,
\[
\nabla_Z X \in \mathfrak{X}(\mathcal{H}), \quad \nabla_Z U \in \mathfrak{X}(\mathcal{V});
\]
\item $\nabla$ is a metric connection: $\nabla g=0$;
\item The torsion $\Tor^\nabla$ of $\nabla$ satisfies $\Tor^\nabla(\mathcal{H},\mathcal{H}) \subset \mathcal{V}$, $\Tor^\nabla(\mathcal{H},\mathcal{V}) \subset \mathcal{V}$ and $\Tor^\nabla(\mathcal{V},\mathcal{V}) =0$.
\item For every $X \in \mathfrak{X}(\mathcal{H})$, $V,Z \in \mathfrak{X}(\mathcal{V})$, 
\begin{equation}\label{symmetry torsion}
 \langle \Tor^\nabla(Z,X), V \rangle =\langle \Tor^\nabla(V,X), Z \rangle.
\end{equation}

\end{enumerate}

\begin{remark}\label{connection project}
Assume that the foliation $\mathcal{F}$ arises from a Riemannian submersion $\pi :(M,g) \to (B,h)$. Let $X,Y \in \mathfrak{X}(M)$ be basic vector fields on $M$, then $\nabla_X Y$ is also basic. In fact, from \cite[Lemma 1.4.1]{Gromoll}, if $\tilde{X},\tilde{Y}$ are the vector fields on $B$ that are $\pi$-related to $X,Y$, then $\nabla_X Y$ is $\pi$-related to $D^B_{\tilde X} \tilde {Y}$, i.e.
\[
\pi_* \left(\nabla_X Y \right) =(D^B_{\tilde X} \tilde {Y}) \circ \pi
\]
where $D^B$ is the Levi-Civita connection on $B$. 
\end{remark}

A short computation shows that the torsion tensor  $\Tor^\nabla$   is  given by
\begin{equation*}
\Tor^\nabla(X,Y) = \begin{cases}
- [X,Y]_\mathcal{V} & X,Y \in \mathfrak{X}(\hor), \\
C_X Y  &  X \in \mathfrak{X}(\hor), Y \in \mathfrak{X}(\V), \\
0 & X,Y \in \mathfrak{X}(\ver).
\end{cases}
\end{equation*}

\begin{remark}
Recall that the foliation is totally geodesic if and only if
\begin{align}\label{total geodesic condition}
 (\mathcal{L}_X g)(U, U)=0, \qquad \text{for any $X \in \mathfrak{X}(\Ho)$, $U \in \mathfrak{X}(\V)$.}
 \end{align}
This is equivalent to $C=0$. Therefore, it follows from the torsion formula that the foliation is totally geodesic if and only if $\Tor^\nabla(\Ho,\V)=0$.
\end{remark}

\begin{example}[Carnot groups]
Consider the foliation on a Carnot group as in Proposition \ref{Example Carnot}. For a left invariant vector field $X$, denote $\mathrm{ad}_X$ the map $\mathrm{ad}_X (Y)=[X,Y]$ and $\mathrm{ad}^*$ its adjoint. It follows from Koszul's formula that for left invariant vector fields
\[
\nabla_X Y =
\begin{cases}
0 &X,Y \in \mathfrak{X}(\mathcal{H}), \\
  0  &X \in \mathfrak{X}(\mathcal{V}),\ Y \in \mathfrak{X}(\mathcal{H}), \\
 \frac{1}{2} \mathrm{ad}_X Y-\frac{1}{2} (\mathrm{ad}^*_X Y)_\V   &X \in \mathfrak{X}(\mathcal{H}),\ Y \in \mathfrak{X}(\mathcal{V}), \\
 \frac{1}{2} \mathrm{ad}_X Y-\frac{1}{2} (\mathrm{ad}^*_Y X)_\V-\frac{1}{2} (\mathrm{ad}^*_X Y)_\V &X,Y \in \mathfrak{X}(\mathcal{V}).
\end{cases}
\]
and therefore, 
\begin{equation*}
\Tor^\nabla(X,Y) = \begin{cases}
- \mathrm{ad}_X Y& X,Y \in \mathfrak{X}(\hor), \\
 -\frac{1}{2} \mathrm{ad}_X Y-\frac{1}{2} (\mathrm{ad}^*_X Y)_\V  &  X \in \mathfrak{X}(\hor), Y \in \mathfrak{X}(\V), \\
0 & X,Y \in \mathfrak{X}(\ver).
\end{cases}
\end{equation*}
\end{example}

For $Z \in \mathfrak{X}(TM)$, there is a  unique skew-symmetric endomorphism  $J_Z:T_xM \to T_xM$ such that for all vector fields $X$ and $Y$,
\begin{align}\label{Jmap}
\left\langle J_Z X,Y\right\rangle= \left\langle Z,\Tor^\nabla (X,Y) \right\rangle.
\end{align}
From the definition and since $\Tor^\nabla$ is vertical, we deduce that  if $Z \in \mathfrak{X}(TM)$, then $J_Z=0$.

The relation between the Levi-Civita connection $D$  and the connection $\nabla$ is given by the formula

\begin{equation}\label{Levi-Civita}
\nabla_X Y=D_X Y+\frac{1}{2} \Tor^\nabla (X,Y)-\frac{1}{2} J_XY-\frac{1}{2} J_Y X, \quad X,Y \in \mathfrak{X}(TM).
\end{equation}

It is a trivial consequence of \eqref{Levi-Civita} that the Riemannian geodesic equation $D_{\dot\gamma} \dot\gamma=0$ becomes
\begin{align}\label{geodesic1}
\nabla_{\dot\gamma} \dot\gamma+ J_{\dot\gamma} \dot\gamma =0.
\end{align}

Thanks to \eqref{def: horizontal laplace}, in a local horizontal frame $X_i$ the horizontal Laplacian $\Dh$ is given by
\begin{align*}
\Dh&=\sum_{i=1}^n X_i^2 -\sum_{i=1}^n (D_{X_i}X_i)_{\mathcal{H}} \\
 &=\sum_{i=1}^n X_i^2 -\sum_{i=1}^n \nabla_{X_i}X_i \\
 &=\sum_{i=1}^n \mathrm{Hess}^\nabla f (X_i,X_i) \\
 &=\mathrm{Tr}_\Ho( \mathrm{Hess}^\nabla f),
\end{align*}
where $\mathrm{Hess}^\nabla f$ denotes the $\nabla$-Hessian of $f$ and $\mathrm{Tr}_\Ho$ the horizontal trace, i.e. the partial trace taken on the horizontal space. Note that from \eqref{Levi-Civita} we also have
\begin{align*}
\Dh&=\sum_{i=1}^n X_i^2 -\sum_{i=1}^n D_{X_i}X_i \\
 &=\mathrm{Tr}_\Ho( \mathrm{Hess}^D f).
\end{align*}
We point out that those formulas for the horizontal Laplacian crucially rely on the assumption of the leaves being minimal.  Without this assumption the formula for $\Dh$ also involves the mean curvature curvature of the leaves, see \cite{AGAG}.

\subsection{The index form in horizontal directions}

Recall that in Riemannian geometry the second variation of the length functional is given by the index form.  The index form of a vector field $X$  along a unit speed geodesic $\gamma : [0,r] \to M $ is defined as 
\begin{align}\label{def:index}
I (\gamma,X):   =\int_0^r \left | D_{\dot \gamma}  X \right|^2 -\left \langle \mathrm{Riem}^D(\dot \gamma,X)X, \dot \gamma \right \rangle  dt
\end{align}
where $\Riem^D$ denotes the Riemann curvature tensor of  the Levi-Civita connection which is defined for $X,Y,Z \in \mathfrak{X} (TM)$ as
\begin{align}\label{riem curv}
\Riem^D(X,Y)Z=D_XD_Y Z -D_Y D_X Z -D_{[X,Y]} Z.
\end{align}

Our  goal is to give a formula for the index form on horizontal fields. As already pointed out the Levi-Civita connection does not preserve horizontality, so a first step is to compute $\Riem^D$ in terms of $\Riem^\nabla$, the Riemann curvature tensor of $\nabla$.

\begin{lemma}\label{Riemann in terms of Hlad}
Let $A$ be the $(2,1)$ tensor defined by
\[
A_X Y=-\frac{1}{2} \Tor^\nabla (X,Y)+\frac{1}{2} J_XY+\frac{1}{2} J_Y X
\]
The curvature operator of the Levi-Civita can be written as 
\begin{align*}
\Riem^{D}(X,Y)Z =&\Riem^{\nabla}(X,Y)Z
+(\nabla_XA)_Y Z-(\nabla_YA)_X Z \\
& + A_{\Tor^{\nabla}(X,Y)} Z
+ A_X A_Y Z - A_Y A_X Z,
\end{align*}
for all vector fields $X,Y,Z \in \mathfrak{X}(TM)$.
\end{lemma}

\begin{proof}
This type of computation is standard, we include it for the sake of completeness. By definition the curvature of $D$ is
\[
\Riem^{D}(X,Y)Z
= D_X(D_YZ)-D_Y(D_XZ)-D_{[X,Y]}Z.
\]
Expand each $D$ using $D=\nabla + A$:
\begin{align*}
D_X(D_YZ)
&= \nabla_X(\nabla_YZ + A_Y Z) + A_X\big(\nabla_YZ + A_Y Z\big) \\
&= \nabla_X\nabla_YZ + \nabla_X\big(A_Y Z\big) + A_X\nabla_YZ + A_X A_Y Z,
\end{align*}
and similarly
\begin{align*}
D_Y(D_XZ)
&= \nabla_Y\nabla_XZ + \nabla_Y\big(A_X Z\big) + A_Y \nabla_XZ + A_Y A_X Z,\\
D_{[X,Y]}Z
&= \nabla_{[X,Y]}Z + A_{[X,Y]} Z.
\end{align*}
Subtracting, we get
\begin{multline*}
\Riem^{D}(X,Y)Z
= \big(\nabla_X\nabla_YZ - \nabla_Y\nabla_XZ - \nabla_{[X,Y]}Z\big) \\
+ \big(\nabla_X\big(A_Y Z\big)-\nabla_Y\big(A_X Z\big)-A_{[X,Y]} Z\big) \\
+ \big(A_X \nabla_YZ - A_Y \nabla_XZ\big)
+ \big(A_X A_Y Z - A_Y A_X Z\big).
\end{multline*}

The first parenthesis is $\Riem^{\nabla}(X,Y)Z$. Expand the second parenthesis using the Leibniz rule:
\begin{align*}
\nabla_X\big(A_Y Z\big)
&= (\nabla_XA)_Y  Z + A_{\nabla_X Y}Z +A_Y \nabla_XZ,\\
\nabla_Y\big(A_X Z\big)
&= (\nabla_YA)_X  Z + A_{\nabla_Y X}Z +A_X \nabla_YZ.
\end{align*}
Hence
\begin{align*}
& \nabla_X\big(A_Y Z\big)-\nabla_Y\big(A_X Z\big)-A_{[X,Y]} Z \\
=&  (\nabla_XA)_Y  Z - (\nabla_YA)_X  Z
+ A_Y \nabla_XZ - A_X \nabla_YZ+A_{\nabla_X Y-\nabla_Y X-[X,Y]}Z.
\end{align*}
Note that by definition the subscript in the last term is $\Tor^{\nabla}(X,Y)$. Substitute this into the expression for $\Riem^D$; the term $A_Y \nabla_XZ - A_X\nabla_YZ$ cancels against the earlier $A_X \nabla_YZ - A_Y \nabla_XZ$ term, leaving
\begin{align*}
\Riem^{D}(X,Y)Z =&\Riem^{\nabla}(X,Y)Z
+(\nabla_XA)_Y Z-(\nabla_YA)_X Z \\
& + A_{\Tor^{\nabla}(X,Y)} Z
+ A_X A_Y Z - A_Y A_X Z.
\end{align*}
\end{proof}

With this lemma at hands, we are now in position to give a formula evaluating the index form \eqref{def:index} on horizontal vectors fields.

\begin{proposition}
Let $X \in \mathfrak{X}(\Ho)$ and $\gamma:[0,r]\to M$ be a unit speed geodesic. Then,
\begin{align*}
I (\gamma,X)   = \int_0^r & \left | \nabla_{\dot \gamma}  X +\frac{1}{2} (J_{\dot \gamma}  X)_\mathcal{H}   \right|^2  -\left \langle \mathrm{Riem}^\nabla(\dot \gamma,X)X, \dot \gamma \right \rangle  +\left\langle (\nabla_X \Tor^\nabla)(X,\dot\gamma),\dot \gamma_\V \right\rangle\\
 &-\frac{1}{4} | J_{\dot \gamma} X|^2_\Ho+ |\Tor^\nabla(X,\dot\gamma)|^2+\left\langle \Tor^\nabla(\Tor^\nabla(\dot\gamma,X),X),\dot \gamma \right\rangle \, dt
\end{align*}
\end{proposition}

\begin{proof}
From the definition \eqref{def:index} we have
\begin{align}\label{demo:index}
I (\gamma,X)   =\int_0^r \left | D_{\dot \gamma}  X \right|^2 -\left \langle \mathrm{Riem}^D(\dot \gamma,X)X, \dot \gamma \right \rangle  dt.
\end{align}
Substitute the formula for $\mathrm{Riem}^D$ given in Lemma \ref{Riemann in terms of Hlad} and  go one by one over the terms appearing in \eqref{demo:index}. The first term is the covariant derivative:
\begin{align*}
| D_{\dot \gamma} X |^2&=  \left|\nabla_{\dot \gamma} X -\frac{1}{2} \Tor^\nabla (\dot \gamma,X)+\frac{1}{2} J_{\dot \gamma} X \right|^2 \\
 &=  \left|\nabla_{\dot \gamma} X +\frac{1}{2} (J_{\dot \gamma} X)_\mathcal{H} \right|^2+\left| -\frac{1}{2} \Tor^\nabla (\dot \gamma,X)+\frac{1}{2} (J_{\dot \gamma} X)_\V \right|^2 \\
 &=\left|\nabla_{\dot \gamma} X +\frac{1}{2} (J_{\dot \gamma} X)_\mathcal{H} \right|^2 +\frac{1}{4}\left| \Tor^\nabla (\dot \gamma,X)\right|^2+\frac{1}{4} \left| (J_{\dot \gamma} X)_\V \right|^2-\frac{1}{2} \left\langle \Tor^\nabla (\dot \gamma,X), J_{\dot \gamma} X \right\rangle \\
 &=\left|\nabla_{\dot \gamma} X +\frac{1}{2} (J_{\dot \gamma} X)_\mathcal{H} \right|^2 +\frac{1}{4}\left| \Tor^\nabla (\dot \gamma,X)\right|^2+\frac{1}{4} \left| (J_{\dot \gamma} X)_\V \right|^2+\frac{1}{2} \left\langle \Tor^\nabla (\Tor^\nabla (\dot \gamma,X),X), \dot \gamma \right\rangle
\end{align*}
After $\Riem^\nabla$ the next three terms are as follows:
\begin{align*}
(\nabla_{\dot \gamma} A)_X X&= \nabla_{\dot \gamma} (A_X X) -A_{\nabla_{\dot \gamma} X}X-A_X \nabla_{\dot \gamma} X \\
 &=\frac{1}{2} \Tor^\nabla (\nabla_{\dot \gamma} X,X)+\frac{1}{2} \Tor^\nabla (X, \nabla_{\dot \gamma} X)=0
\end{align*}
\begin{align*}
\left\langle (\nabla_{X} A)_{\dot \gamma} X,\dot \gamma \right\rangle&=\left\langle-\frac{1}{2} (\nabla_X \Tor^\nabla) (\dot \gamma ,X)+\frac{1}{2} (\nabla_X J)_{\dot \gamma} X, \dot \gamma \right\rangle\\
 &=-\left\langle(\nabla_X \Tor^\nabla) (\dot \gamma ,X), \dot \gamma \right\rangle
\end{align*}
\begin{align*}
\left\langle A_{\Tor^\nabla (\dot \gamma, X)} X, \dot \gamma \right\rangle&=\left\langle -\frac{1}{2}\Tor^\nabla(\Tor^\nabla(\dot \gamma,X),X)+\frac{1}{2} J_{\Tor^\nabla(\dot \gamma,X)} X , \dot \gamma \right\rangle \\
 &=-\frac{1}{2}\left\langle \Tor^\nabla(\Tor^\nabla(\dot \gamma,X),X), \dot \gamma \right\rangle-\frac{1}{2} | \Tor^\nabla(\dot \gamma,X)|^2. 
\end{align*}
Finally, for the last term we get:
\begin{align*}
 & \left\langle (A_{\dot \gamma} A_X -A_X A_{\dot \gamma})X , \dot \gamma \right\rangle \\
 =&-\left\langle A_X A_{\dot \gamma}X, \dot \gamma \right\rangle  \\
 =& \frac{1}{2} \left\langle A_X  \Tor^\nabla(\dot \gamma,X),\dot\gamma \right\rangle -\frac{1}{2} \left\langle A_X J_{\dot \gamma} X,\dot \gamma \right\rangle \\
 =& -\frac{1}{4} \left\langle  \Tor^\nabla(X, \Tor^\nabla(\dot \gamma,X)),\dot \gamma \right\rangle+\frac{1}{4}\left\langle J_{ \Tor^\nabla(\dot \gamma,X)}X,\dot \gamma \right\rangle +\frac{1}{4}  \left\langle \Tor^\nabla(X, J_{\dot \gamma}X),\dot\gamma\right\rangle-\frac{1}{4}\left\langle J_{ J_{\dot \gamma} X}X,\dot\gamma \right\rangle \\
 =& \frac{1}{4} \left\langle \Tor^\nabla(\Tor^\nabla(\dot \gamma,X),X), \dot \gamma \right\rangle-\frac{1}{4}|  \Tor^\nabla(\dot \gamma,X)|^2 +\frac{1}{4}  |J_{\dot \gamma}X|^2-\frac{1}{4} \left\langle \Tor^\nabla(\Tor^\nabla(\dot \gamma,X),X), \dot \gamma \right\rangle \\
 =&-\frac{1}{4}|  \Tor^\nabla(\dot \gamma,X)|^2 +\frac{1}{4}  |J_{\dot \gamma}X|^2.
\end{align*}
Collecting those terms yields our formula.
\end{proof}

\subsection{The tensor $\mathfrak{R}$}

We now introduce the relevant \textit{Ricci-like} tensor which will be the main character in our analysis of the horizontal Laplacian comparison theorems and the goal of the section is interpret this tensor geometrically. Define a $(2,1)$ tensor by the formula: For $U,V \in \mathfrak{X}(TM)$
\begin{align*}
\mathfrak{R}(U,V)= \sum_{i=1}^n & \left \langle \mathrm{Riem}^\nabla(U,X_i)X_i, V \right \rangle  -\left\langle (\nabla_{X_i} \Tor^\nabla)(X_i,U), V \right\rangle-\left\langle \Tor^\nabla(\Tor^\nabla(U,X_i),X_i),V \right\rangle  \\
 & +\frac{1}{4} \langle J_{U} X_i,J_V X_i\rangle_\Ho- \langle \Tor^\nabla(X_i,U),\Tor^\nabla(X_i,V)\rangle  
\end{align*}
where the $X_i$'s form an arbitrary orthonormal frame of horizontal vectors. Note that it is clear from the definition that $\mathfrak{R}(U,V)$ does not depend on the choice of the $X_i$'s.

\begin{lemma}\label{sym Ricci}
For $U,V \in \mathfrak{X}(TM)$ and $X \in \mathfrak{X}(\Ho)$
\[
\left\langle \mathrm{Riem}^\nabla(V,X)X, U \right \rangle=\left\langle \mathrm{Riem}^\nabla(U,X)X, V \right \rangle=\left \langle \mathrm{Riem}^\nabla(U_\Ho,X)X, V_\Ho \right \rangle.
\]
\end{lemma}
\begin{proof}
Since $\nabla$ preserves horizontality, we have $\mathrm{Riem}^\nabla(U,X)X \in  \mathfrak{X}(\Ho)$ and therefore
\begin{align*}
 \langle \mathrm{Riem}^\nabla(U,X)X, V \rangle= \langle \mathrm{Riem}^\nabla(U,X)X,  V_\mathcal{H} \rangle.
\end{align*}
 Using then the fact that $  \mathrm{Riem}^\nabla(U_\V,X)$ is skew-symmetric  we get
\begin{align*}
 \langle \mathrm{Riem}^\nabla(U_\V,X)X, V_\Ho \rangle&=- \langle \mathrm{Riem}^\nabla(U_\V,X)V_\Ho, X \rangle.
\end{align*}
Recall then the Bianchi identity for connections with torsion \cite[Theorem 1.24]{Besse}: For arbitrary vector fields $A,B,C\in\mathfrak{X}(TM)$.
\begin{equation}\label{eq:bianchi}
  \circlearrowleft \Riem^\nabla(A,B)C
  \;=\;
  \circlearrowleft (\nabla_A \Tor^\nabla)(B,C)
  \;+\;
  \circlearrowleft \Tor^\nabla\big(\Tor^\nabla(A,B),C\big),
\end{equation}
where  $\circlearrowleft$ denotes the cyclic sum.  Using the fact that the torsion is always vertical yields
\[
\langle \mathrm{Riem}^\nabla(U_\V,X)V_\Ho, X \rangle=-\langle \mathrm{Riem}^\nabla(X, V_\Ho)U_\V, X \rangle-\langle \mathrm{Riem}^\nabla( V_\Ho,U_\V)X, X \rangle.
\]
Since $ \mathrm{Riem}^\nabla(X, V_\Ho)U_\V$ is vertical and $\mathrm{Riem}^\nabla(V_\Ho,U_\V)$ skew-symmetric we infer
\[
 \langle \mathrm{Riem}^\nabla(U_\V,X)X, V_\Ho \rangle=0
\]
and conclude
\begin{align*}
 \langle \mathrm{Riem}^\nabla(U,X)X, V \rangle= \langle \mathrm{Riem}^\nabla(U_\Ho,X)X,  V_\mathcal{H} \rangle.
\end{align*}
The symmetry 
\[
\left\langle \mathrm{Riem}^\nabla(V,X)X, U \right \rangle=\left\langle \mathrm{Riem}^\nabla(U,X)X, V \right \rangle
\]
follows from the same type of computations: first use the skew-symmetry of $\mathrm{Riem}^\nabla$ and then Bianchi's identity.
\end{proof}

In order to better understand  $\mathfrak{R}$ we introduce the following definitions/notations.  As before, the $X_i$'s below form an arbitrary orthonormal frame of horizontal vectors and $U, V$ are arbitrary vectors  in $\mathcal{X}(TM)$.

\begin{itemize}
\item The horizontal Ricci curvature if the connection $\nabla$ is defined as the $(2,0)$ tensor
\[
\mathrm{Ric}^\nabla_\Ho(U,V)=\sum_{i=1}^n  \left \langle \mathrm{Riem}^\nabla(U,X_i)X_i, V \right.\rangle
\]
\item The horizontal divergence of the torsion is defined as the $(1,1)$ tensor
\[
\delta^\nabla_\Ho T (U)= \sum_{i=1}^n \nabla_{X_i} \Tor^\nabla(X_i,U).
\]
\item
\[
(\Tor^\nabla_U ,\Tor^\nabla_V)_\Ho=\sum_{i=1}^n \left\langle \Tor^\nabla (U,X_i), \Tor^\nabla (V,X_i)\right\rangle_\Ho
\]
\item 
\[
(J_U ,J_V)_\Ho=\sum_{i=1}^n \left\langle J_U X_i , J_V X_i \right\rangle_\Ho
\]
\end{itemize}

\begin{remark}
Assume that the foliation $\mathcal{F}$ arises from a Riemannian submersion $\pi :(M,g) \to (B,h)$. Let $X,Y \in \mathfrak{X}(\Ho)$ be basic vector fields on $M$, then it follows from Remark \ref{connection project} that
\[
\mathrm{Ric}^\nabla_\Ho(X,Y)=\mathrm{Ric}_B(\pi_*X,\pi_*Y)
\]
where $\mathrm{Ric}_B$ is the Ricci curvature of the base manifold $(B,h)$.
\end{remark}

With those definitions, we can decompose $\mathfrak{R}$ as follows.

\begin{proposition}\label{example carnot r}
\
\begin{enumerate}
\item For $X,Y \in \mathfrak{X}(\Ho)$,
\[
\mathfrak{R}(X,Y)=\mathfrak{R}(Y,X)=\mathrm{Ric}^\nabla_\Ho(X,Y)-(\Tor^\nabla_X ,\Tor^\nabla_Y)_\Ho .
\]
\item For $U \in \mathfrak{X}(\V)$, $X \in \mathfrak{X}(\Ho)$,
\[
\mathfrak{R}(U,X)=-(\Tor^\nabla_U ,\Tor^\nabla_X)_\Ho .
\]
\item For $U \in \mathfrak{X}(\V)$, $X \in \mathfrak{X}(\Ho)$,
\[
\mathfrak{R}(X,U)=-\left\langle \delta_\Ho^\nabla \Tor^\nabla (X),U \right\rangle-2(\Tor^\nabla_U ,\Tor^\nabla_X)_\Ho .
\]
\item  For $U,V \in \mathfrak{X}(\V)$, 
\[
\mathfrak{R}(U,V)=\mathfrak{R}(V,U)=-\left\langle \delta_\Ho^\nabla \Tor^\nabla (U),V \right\rangle+\frac{1}{4}(J_U,J_V)_\Ho-2(\Tor^\nabla_U ,\Tor^\nabla_V)_\Ho .
\]

\end{enumerate}
\end{proposition}

\begin{proof}

\

\begin{enumerate}
\item This follows from the definition of $\mathrm{Ric}^\nabla_\Ho$, Lemma \ref{sym Ricci} and the fact that $\Tor^\nabla$ is always vertical.
\item This is immediate from Lemma \ref{sym Ricci} and the fact that $\Tor^\nabla$ is always vertical.
\item We have from \eqref{symmetry torsion} 
\begin{align*}
\sum_{i=1}^n \left\langle \Tor^\nabla(\Tor^\nabla(X,X_i),X_i),U \right\rangle =\sum_{i=1}^n \left\langle \Tor^\nabla(X,X_i),\Tor^\nabla(U,X_i) =(\Tor^\nabla_U ,\Tor^\nabla_X)_\Ho\right\rangle
\end{align*}
and the other terms follow either from the definition or as before.
\item The computation is similar to the others. Note that the symmetry $\mathfrak{R}(U,V)=\mathfrak{R}(V,U)$ also follows from \eqref{symmetry torsion}.
\end{enumerate}
\end{proof}

The following corollary is straightforward.

\begin{corollary}
The tensor $\mathfrak{R}$ is symmetric, i.e. $\mathfrak{R}(U,V)=\mathfrak{R}(V,U)$ for every $U,V \in \mathfrak{X}(TM)$ if and only if the horizontal distribution is Yang-Mills, meaning that for every $U \in \mathfrak{X}(\V)$, $X \in \mathfrak{X}(\Ho)$,
\begin{align}\label{YM condition}
\left\langle \delta_\Ho^\nabla \Tor^\nabla (X),U \right\rangle+(\Tor^\nabla_U ,\Tor^\nabla_X)_\Ho=0 .
\end{align}
\end{corollary}

The terminology Yang-Mills for \eqref{YM condition} originates from \cite[Definition 9.35]{Besse}. Notice that this condition is intrinsically sub-Riemannian, i.e. it only depends on $\mathcal{H}$ and the restriction of $g$ to $\mathcal{H}$ and not on the metric on the leaves.

\begin{example}[Carnot groups]
Consider the foliation on a Carnot group as in Proposition \ref{Example Carnot}. Computations show that:
\begin{enumerate}
\item For $X,Y \in \mathfrak{X}(\Ho)$,
\[
\mathfrak{R}(X,Y)=-\sum_{i=1}^n \left\langle \mathrm{ad}_{X_i} (X),  \mathrm{ad}_{X_i} (Y) \right\rangle.
\]
\item For $U \in \mathfrak{X}(\V)$, $X \in \mathfrak{X}(\Ho)$,
\[
\mathfrak{R}(U,X)=-\frac{1}{2} \sum_{i=1}^n \left\langle \mathrm{ad}_{X_i} (X),  \mathrm{ad}^*_{X_i} (U) \right\rangle .
\]
\item For $U \in \mathfrak{X}(\V)$, $X \in \mathfrak{X}(\Ho)$,
\[
\mathfrak{R}(X,U)=-\frac{1}{2} \sum_{i=1}^n \left\langle \mathrm{ad}_{X_i} (X),  \mathrm{ad}^*_{X_i} (U) \right\rangle .
\]
\item  For $U,V \in \mathfrak{X}(\V)$, 
\begin{align*}
\mathfrak{R}(U,V)= &\frac{1}{4}(J_U,J_V)_\Ho-\sum_{i=1}^n \left\langle \mathrm{ad}_{X_i} (U),  \mathrm{ad}_{X_i} (V) \right\rangle -\frac{1}{2}\sum_{i=1}^n \left\langle \mathrm{ad}_{X_i} (U),  \mathrm{ad}^*_{X_i} (V) \right\rangle \\
 &-\frac{1}{2}\sum_{i=1}^n \left\langle \mathrm{ad}_{X_i} (U),  \mathrm{ad}^*_{X_i} (V) \right\rangle.
\end{align*}
\end{enumerate}
In particular, in any Carnot group the horizontal distribution if Yang-Mills and there exists $K \le 0$ such that for every $U \in \mathfrak{X}(TM)$, $\mathfrak{R}(U,U) \ge K |U|^2$.
\end{example}

To conclude this section, we point out how the formulas simplify when the foliation is totally geodesic. The computations are straightforward and follow from the fact that on totally geodesic foliations $\Tor^\nabla (\Ho,\V) =0$.

\begin{corollary}
Assume that the foliation is totally geodesic.
\begin{enumerate}
\item For $X,Y \in \mathfrak{X}(\Ho)$,
\[
\mathfrak{R}(X,Y)=\mathrm{Ric}^\nabla_\Ho(X,Y)-(\Tor^\nabla_X ,\Tor^\nabla_X)_\Ho .
\]
\item For $U \in \mathfrak{X}(\V)$, $X \in \mathfrak{X}(\Ho)$,
\[
\mathfrak{R}(U,X)=0 .
\]
\item For $U \in \mathfrak{X}(\V)$, $X \in \mathfrak{X}(\Ho)$,
\[
\mathfrak{R}(X,U)=-\left\langle \delta_\Ho^\nabla \Tor^\nabla (X),U \right\rangle .
\]
\item  For $U,V \in \mathfrak{X}(\V)$, 
\[
\mathfrak{R}(U,V)=\frac{1}{4}(J_U,J_V)_\Ho .
\]

\end{enumerate}

\end{corollary}

\subsection{Sub-Laplacian comparisons}

We now turn to the comparison theorems.

\begin{theorem}\label{comparison laplacian}
Let $p \in M$ and denote $r_p(x)=d(p,x)$. Assume that there exists $K \in \mathbb{R}$ such that for every $U \in \mathfrak{X}(TM)$
\[
\mathfrak{R}(U,U) \ge K |U|^2.
\] 
Then, for $x\neq p$ not in the  cut-locus of $p$,
\begin{align*}
\Delta_\Ho r_p (x) \le 
\begin{cases}
 \sqrt { nK} \cot\left(\sqrt { \frac{K}{n}} r_p(x) \right) &\text{if}\ K>0,
\\
\displaystyle\frac{n}{r_p(x)} &\text{if}\ K = 0,
\\
 \sqrt{n |K|} \coth\left(\sqrt{\frac{|K|}{n}} r_p(x)\right) &\text{if}\ K<0.
\end{cases}
\end{align*}

\end{theorem}

\begin{proof}
Let $\gamma:[0,r]\to M$ be the unique unit speed  geodesic between $p$ and $x$.  Note that since $\gamma$ has unit speed, $r=r_p(x)$. Pick an arbitrary orthonormal horizontal frame $v_1,\cdots,v_n$ at $x$. We have at $x$
\[
\Delta_\Ho r_p =\sum_{i=1}^n \mathrm{Hess}^\nabla r_p (v_i,v_i)  .
\]

Let $X_1,\ldots,X_n$ be the family of horizontal vector fields  along $\gamma$ such that
\begin{align*}
\begin{cases}
\nabla_{\dot \gamma}X_i+\frac{1}{2} (J_{\dot \gamma} X_i)_\mathcal{H}=0 \\
X_i(0)=v_i.
\end{cases}
\end{align*}
Note that we have
\begin{align}\label{skew connection}
\nabla_{\dot \gamma} \left\langle X_i, X_j \right\rangle& =  \left\langle \nabla_{\dot \gamma} X_i, X_j \right\rangle+\left\langle  X_i, \nabla_{\dot \gamma}X_j \right\rangle \\ \notag
 &=-\frac{1}{2}\left( \left\langle (J_{\dot \gamma} X_i)_\Ho, X_j \right\rangle+\left\langle  X_i, (J_{\dot \gamma} X_j)_\Ho\right\rangle\right) \\ \notag
 &=-\frac{1}{2}\left( \left\langle J_{\dot \gamma} X_i, X_j \right\rangle+\left\langle  X_i, J_{\dot \gamma} X_j\right\rangle\right) =0.
\end{align}
Therefore $X_1,\ldots,X_n$ is an orthonormal frame along $\gamma$.
Consider now the vector fields along $\gamma$ defined by
\[
Y_i(t)=\frac{\mathfrak{s}_K(t)}{\mathfrak{s}_K(r)} X_i(t), \quad 0\le t \le r,
\]
where
\begin{align*}
\mathfrak{s}_K(t)= 
\begin{cases}
 \sin \left(\sqrt { \frac{K}{n}} t \right) &\text{if}\ K>0,
\\
\displaystyle t &\text{if}\ K = 0,
\\
\sinh \left(\sqrt{\frac{|K|}{n}} t\right) &\text{if}\ K<0.
\end{cases}
\end{align*}

We have
\[
Y_i(0)=0, Y_i(r)=v_i.
\]

Using the index lemma and the fact that the $X_i$ form a horizontal orthonormal frame, one has
\begin{align*}
\Delta_\Ho r_p (x) & \le \sum_{i=1}^n I(\gamma,Y_i) \\
  &=  \sum_{i=1}^n  \int_0^r  \left | \nabla_{\dot \gamma}  Y_i +\frac{1}{2} (J_{\dot \gamma}  Y_i)_\mathcal{H}   \right|^2  -\left \langle \mathrm{Riem}^\nabla(\dot \gamma,Y_i)Y_i, \dot \gamma \right \rangle  +\left\langle (\nabla_{Y_i} \Tor^\nabla)(Y_i,\dot\gamma),\dot \gamma_\V \right\rangle\\
 &\quad -\frac{1}{4} | J_{\dot \gamma} Y_i |^2_\Ho+ |\Tor^\nabla(Y_i,\dot\gamma)|^2+\left\langle \Tor^\nabla(\Tor^\nabla(\dot\gamma,Y_i),Y_i),\dot \gamma \right\rangle \, dt \\
 &=\frac{1}{\mathfrak{s}_K(r)^2}\sum_{i=1}^n  \int_0^r  \mathfrak{s}'_K(t)^2  - \mathfrak{s}_K(t)^2 \left(  \left \langle \mathrm{Riem}^\nabla(\dot \gamma,X_i)X_i, \dot \gamma \right \rangle  +\left\langle (\nabla_{X_i} \Tor^\nabla)(X_i,\dot\gamma),\dot \gamma_\V \right\rangle \right. \\
 &\left. \quad -\frac{1}{4} | J_{\dot \gamma} X_i |^2_\Ho+ |\Tor^\nabla(X_i,\dot\gamma)|^2+\left\langle \Tor^\nabla(\Tor^\nabla(\dot\gamma,X_i),X_i),\dot \gamma \right\rangle \right) \, dt \\
 &=\frac{1}{\mathfrak{s}_K(r)^2}  \int_0^r n  \mathfrak{s}'_K(t)^2-\mathfrak{s}_K(t)^2 \mathfrak{R}(\dot \gamma,\dot \gamma)  dt \\
 &\le \frac{1}{\mathfrak{s}_K(r)^2}  \int_0^r n  \mathfrak{s}'_K(t)^2- K \mathfrak{s}_K(t)^2  dt.
\end{align*}
The result follows after a straightforward computations.
\end{proof}

Besides the sub-Laplacian comparison theorem above, we will also need the following comparison result for the trace of second variation of the length of geodesics with  \textit{skewed $\nabla$-parallel} deformations at the endpoints. We first need to introduce some notations. Denote
\begin{align}\label{bi cut-locus}
\mathcal{C} =\left\{ (p,q) \in M \times M \mid p\neq q, \, p \text{  and  } q \text{ are not in the cut locus of one another} \, \right\}.
\end{align}
Let $v_i$ be  an arbitrary horizontal orthonormal frame at $p$ and let  $\tilde{v}_i=X_i(r)$ where the $X_i$'s are the family of horizontal vector fields  along the unique unit speed geodesic $\gamma:[0,r] \to M$ connecting $p$ to $q$ and such that
\begin{align*}
\begin{cases}
\nabla_{\dot \gamma}X_i+\frac{1}{2} (J_{\dot \gamma} X_i)_\mathcal{H}=0 \\
X_i(0)=v_i.
\end{cases}
\end{align*}
It follows from the computation \eqref{skew connection} that $\tilde{v}_i$ is a orthonormal frame at $q$.

We then consider on $\mathcal{C}$ a horizontal Laplacian with \textit{skewed $\nabla$-parallel coupling} defined for a smooth function $f :\mathcal{C} \to \mathbb R$ as
\begin{align}\label{coupled laplacian}
\Delta^c_{\Ho,\Ho} f (p,q)=\Delta^1_{\Ho} f (p,q)+\Delta^2_{\Ho} f (p,q) +2 \sum_{i=1}^n v_i^1 \tilde{v}_i^2 f(p,q)
\end{align}

In the formula for $\Delta^c_{\Ho,\Ho}$, the superscript $1$ or $2$ means the action on the first or the second set of $M$ variables in the product $M\times M$. Also $v_i^1 \tilde{v}_i^2 f(p,q)$ is defined as $V_i^1 \tilde{V}_i^2 f (p,q)$ where $V_i$ (resp. $\tilde{V}_i$) is any vector field that coincides with $v_i$ (resp. $\tilde{v}_i$) at $p$ (resp. $q$). Finally, note that the term $ \sum_{i=1}^n v_i^1 \tilde{v}_i^2 f$  does not depend on the choice of the frame $v_i$. 

\begin{theorem}\label{comparison coupled laplacian}
 Assume that there exists $K \in \mathbb{R}$ such that for every $U \in \mathfrak{X}(TM)$
\[
\mathfrak{R}(U,U) \ge K |U|^2.
\] 
Then, for $(p,q) \in \mathcal{C}$
\begin{align*}
\Delta^c_{\Ho,\Ho} d (p,q) \le 
\begin{cases}
 -2n\sqrt{K} \tan \left(\sqrt{\frac{K}{n}} \frac{d(p,q)}{2} \right) &\text{if}\ K>0,
\\
0&\text{if}\ K = 0,
\\
 2n\sqrt{|K|} \tanh \left(\sqrt{\frac{|K|}{n}} \frac{d(p,q)}{2} \right) &\text{if}\ K<0.
\end{cases}
\end{align*}
\end{theorem}

\begin{proof}
Consider geodesic curves $\sigma_i:(-\varepsilon,\varepsilon) \to M$, $\tilde{\sigma}_i:(-\varepsilon, \varepsilon) \to M$ such that
\[
\sigma_i(0)=p, \,  \dot{\sigma}_i (0)=v_i
\]
and
\[
\tilde{\sigma}_i(0)=q, \,  \dot{\tilde{\sigma}}_i (0)=\tilde{v}_i.
\]
We have from the chain rule
\[
\Delta^c_{\Ho,\Ho} d (p,q) =\sum_{i=1}^{n} \left( \frac{d^2}{ds^2} \right)_{ s=0}  d( \sigma_i(s), \tilde{\sigma}_i(s)).
\]
Therefore, from the second variation formula \cite[Theorem 10.12]{Lee} and the index lemma one deduces that
\[
\Delta^c_{\Ho,\Ho} d (p,q) \le  \sum_{i=1}^n I(\gamma,Y_i)
\]
where the $Y_i$'s are arbitrary vector fields along $\gamma$ such that
\[
Y_i(0)=v_i, \, \, Y_i(r)=\tilde{v}_i.
\]
Pick
\[
Y_i(t)=\mathfrak{c}_K (t) X_i(t),
\]
where
\begin{align*}
\mathfrak{c}_K(t)= 
\begin{cases}
 \cos \left(\sqrt { \frac{K}{n}} t \right)  +\frac{1-\cos \left(\sqrt { \frac{K}{n}} t \right)}{\sin \left(\sqrt { \frac{K}{n}} r \right)}\sin \left(\sqrt { \frac{K}{n}} t \right)&\text{if}\ K>0,
\\
\displaystyle 1 &\text{if}\ K = 0,
\\
 \cosh \left(\sqrt { \frac{K}{n}} t \right)  +\frac{1-\cosh \left(\sqrt { \frac{K}{n}} t \right)}{\sinh \left(\sqrt { \frac{K}{n}} r \right)}\sinh \left(\sqrt { \frac{K}{n}} t \right) &\text{if}\ K<0.
\end{cases}
\end{align*}
By a computation similar to the one in the proof of Theorem \ref{comparison laplacian}, we obtain
\begin{align*}
\sum_{i=1}^n I(\gamma,Y_i) \le \int_0^r n  \mathfrak{c}'_K(t)^2- K \mathfrak{c}_K(t)^2  dt
\end{align*}
and the conclusion follows from evaluating the above integral.
\end{proof}

\subsection{A Bonnet-Myers type theorem}

A classical application of Laplacian comparison is the Bonnet-Myers. In our setting we get the following result.

\begin{corollary}[Bonnet-Myers type theorem]\label{BMyers}
Assume that there exists $K >0$ such that for every $U \in \mathfrak{X}(TM)$
\[
\mathfrak{R}(U,U) \ge K |U|^2,
\] 
then $M$ is compact and 
\[
\mathbf{diam} (M ) \le \pi \sqrt{ \frac{n}{K}}.
\]
\end{corollary}

\begin{proof}
Let $p \in M$. From Theorem \ref{comparison laplacian}, one has for $x\neq p $, not in the cut-locus of $p$
\begin{equation*}
\Delta_{\mathcal{H}} r_{p}(x) \le \sqrt {nK} \cot \left(\sqrt {\frac{K}{n}} r_{p}(x)\right).
\end{equation*}
We deduce from Calabi's lemma that any point $x$ such that $d (p,x) \ge \pi \sqrt{ \frac{n}{K}}$ has to be in the cut-locus of $p$.
Let now $x \in \M$ arbitrary. If $x$ is not in the cut-locus of $p$, then $d(p,x) < \pi \sqrt{ \frac{n}{K}}$. If $x$ is in the cut-locus of $p$ then for every $\eta >0$ there is at least one point $y$ in the open ball with center $x$ and radius $\eta$ such that $y$ is not in the cut-locus of $p$. Thus $d(p , x)\le \pi \sqrt{ \frac{n}{K}} +\eta$. We conclude $d(p , x)\le \pi \sqrt{ \frac{n}{K}}$ since $\eta$ is arbitrary.
\end{proof}

\section{Radial processes, parallel couplings and  applications to the horizontal semigroup}

Let $( ( \xi_t )_{t \ge 0} , ( \mathbb P_x )_{x \in M} )$ be the subelliptic
diffusion process generated by $\Delta_{\ch}$ and let $\zeta$ denote its
lifetime. We will refer to~$\xi$
as the horizontal Brownian motion of the foliation or as the
sub-Riemannian Brownian motion. Note that our horizontal Brownian motion is normalized to have $\Delta_{\ch}$ as its generator, rather than $\frac{1}{2}\Delta_{\ch}$.  By the hypoellipticity of $\Delta_{\ch}$, $\xi$ admits a smooth heat
kernel $p_t (x,y)$. The horizontal heat semigroup with generator $\Delta_{\ch}$ is given by
\[
P_t f(x)=\int_M p_t (x,y) f(y) d\mu(y)=\mathbb{E}_x \left( f(\xi_t) 1_{t < \zeta} \right) 
\]
where $f: M\to \mathbb{R}$ is a bounded Borel function and we recall that $\mu$ is the Riemannian volume measure.

\subsection{Radial part of the horizontal Brownian motion and stochastic completeness}
The following theorem was proved in \cite{radial}.

\begin{theorem} \label{th:Ito-radial} 

Take $p \in M$ and as before set $r_p (x) = d ( p , x )$. For each $x_1 \in M$, if $\xi_0 = x_1$, then there
  exists a non-decreasing continuous process $l_t$ which increases only
  when $\xi_t$ is in the cut-locus of $p$ and a martingale $\be_t$ on $\mathbb R$ with quadratic variation satisfying $d\ang{\beta} \le 2 \,dt$ such that
  \begin{equation} \label{eq:Ito-radial} 
  r_p ( \xi_{ t \wedge \zeta } ) = r_p ( x_1 ) + \be_{ t \wedge \zeta} + \int_0^{ t \wedge \zeta }
    \Delta_{\ch} r_p ( \xi_s ) ds - l_{ t \wedge \zeta}
  \end{equation}
  holds $\mathbb P_{x_1}$-almost surely.
\end{theorem}

\begin{theorem}
 Assume that there exists $K \in \mathbb{R}$ such that for every $U \in \mathfrak{X}(TM)$
\[
\mathfrak{R}(U,U) \ge K |U|^2.
\] 
Then for every $x \in M$, $\mathbb{P}_x (\zeta <+\infty)=1$. In particular, the semigroup $P_t$ is stochastically complete meaning that for every $x \in M$ and $t \ge 0$
\[
P_t 1(x)=1.
\] 
\end{theorem}

\begin{proof}
We can assume $K<0$. As before, we fix a point $p \in \M$. For $x_1 \in \M$, we consider
the solution of the stochastic differential equation
\begin{equation} \label{tildeXi}
  \tilde{r}_t =d (p,x_1)+ \sqrt{n |K|} \int_0^t \coth\left(\sqrt{\frac{|K|}{n}} \tilde{r}_t\right)  ds+\beta_t
\end{equation}
where $\beta$ is the martingale defined in Theorem
\ref{th:Ito-radial}.  It follows from Theorem \ref{th:Ito-radial} and the Ikeda-Watanabe
  comparison theorem \cite{Ik-Wa} that for $t < \zeta$, one has
  $\mathbb P_{x_1}$ a.s.
  \[
    d (x_0,\xi_t) \le \tilde{r}_t.
  \]
  The result follows  since $\tilde{r}_t $ can not explode given it has a quadratic variation less than $2t$ and  the drift term  is controlled because $\coth (r)$ is uniformly bounded for large $r$.
  \end{proof}
  
  Using the moment estimates method in \cite{MR4091106} we can also get exit time estimates for the horizontal Brownian motion.
  
  \begin{proposition}
  Assume that there exists $K \le 0$ such that for every $U \in \mathfrak{X}(TM)$
\[
\mathfrak{R}(U,U) \ge K |U|^2.
\] 
There exist constants $c_1,c_2,c_3>0$ such that for every $x \in M$, $r \ge 0$ and $t>0$
\begin{align*}
\mathbb{P}_x \left(\sup_{s\in [0,t]} d(x,\xi_t) \ge r \right) \le 
\begin{cases}
c_1 \exp\left(-c_2 \frac{ \sqrt{|K|}}{e^{c_3\sqrt{|K|}t}-1} r^2 \right) \text{ if } K<0 \\ 
c_1 \exp\left(-c_2 \frac{r^2}{t} \right) \text{ if } K=0.
\end{cases}
\end{align*}
  \end{proposition}
  
  \begin{proof}
  Proceed as in Theorem 1.6 in \cite{MR4091106}; the context is a little different there but all arguments go through.
  \end{proof}
  
  Note that the stochastic completeness and exit time estimate yields on diagonal heat kernel estimates.   

  \begin{corollary}
Assume that there exists $K \le 0$ such that for every $U \in \mathfrak{X}(TM)$
\[
\mathfrak{R}(U,U) \ge K |U|^2.
\] 
Then, there exist  constants $c_1,c_2,c_3>0$ such that for every $x \in M$ and $t>0$
\begin{align*}
p_t(x,x) \ge
\begin{cases}
 \frac{c_1}{\mu(B(x ,c_2\sqrt{t}))}, \text{ if } K=0, \\
 \frac{c_1}{\mu\left(B\left(x ,c_2\left(\frac{e^{c_3\sqrt{|K|}t}-1}{ \sqrt{|K|}} \right)^{1/2}\right)\right)}, \text{ if } K<0.
 \end{cases}
\end{align*}
\end{corollary}
\begin{proof}
For simplicity, we show the argument in the case $K=0$, the case $K<0$ being identical. For $r >0$ and $x \in M$ denote $B(x,r)$ the open ball with center $x$ and radius $r$. We have then
\begin{align*}
\int_{M \setminus B(x,r)} p_t (x,y) d\mu(y) &=\mathbb{P}_x ( \xi_t \notin B(x,r)) \\
 &\le \mathbb{P}_x \left(\sup_{s\in [0,t]} d(x,\xi_t) \ge r \right) \\
 &\le c_1 \exp\left(-c_2 \frac{r^2}{t} \right)
\end{align*}
Then, from stochastic completeness we have
\begin{align*}
 \int_{ B(x,r)} p_t (x,y) d\mu(y)=1-\int_{M \setminus B(x,r)} p_t (x,y) d\mu(y)\ge 1- c_1 \exp\left(-c_2 \frac{r^2}{t} \right).
\end{align*}
  Apply Cauchy-Schwarz inequality
  \begin{align*}
 \left( \int_{ B(x,r)} p_t (x,y) d\mu(y)\right)^2 & \le \mu \left( B(x,r) \right) \int_{ B(x,r)} p_t (x,y)^2 d\mu(y) \\
   & \le  \mu \left( B(x,r) \right) \int_{ M} p_t (x,y)^2 d\mu(y) \\
   &= \mu \left( B(x,r) \right) p_{2t}(x,x).
  \end{align*}
  Therefore we obtained
  \[
  p_{2t}(x,x) \ge \left(1- c_1 \exp\left(-c_2 \frac{r^2}{t} \right)\right)\frac{1}{\mu \left( B(x,r) \right)}.
  \]
  Choosing $r=A\sqrt{t}$ with $A$ large enough yields the conclusion.
  \end{proof}
  
  \subsection{Lipschitz bound on the heat semigroup and coupling by parallel transport}
  
  Given  a Lipschitz function $f : M \to \mathbb{R}$ we denote
  \[
  \mathrm{Lip} (f)=\sup_{x,y \in M, x \neq y}  \frac{|f(x)-f(y)|}{d(x,y)}.
  \]
  and $\mathbf{Lip}_b (M)$ the set of bounded Lipschitz functions on $M$.
\begin{theorem}
 Assume that there exists $K \in \mathbb{R}$ such that for every $U \in \mathfrak{X}(TM)$
\[
\mathfrak{R}(U,U) \ge K |U|^2.
\] 
Then for every $f \in \mathbf{Lip}_b (M)$ and $t \ge 0$, $P_tf \in \mathbf{Lip}_b (M)$ and we have
\[
\mathrm{Lip} (P_t f)\le e^{-Kt} \mathrm{Lip} ( f).
\] 
\end{theorem}

\begin{proof}
The proof relies on a coupling by parallel transport. Define the coupling map as follows. Recall the definition of $\mathcal C$ given by \eqref{bi cut-locus}. For $p,q \in \mathcal{C}$ define the transport map $\mathfrak{P}_{p \to q}:T_pM \to T_qM$ given for $v \in T_pM$ by $\mathfrak{P}_{p \to q}(v)=X(r)$ where $X$ is the horizontal vector field along the unit speed geodesic $\gamma:[0,r] \to M$ with endpoints $p,q$ such that
\begin{align*}
\begin{cases}
\nabla_{\dot \gamma}X+\frac{1}{2} (J_{\dot \gamma} X)_\mathcal{H}=0 \\
X(0)=v.
\end{cases}
\end{align*}
It follows from the computation \eqref{skew connection} that $\mathfrak{P}_{p \to q}$ is an isometry $\mathcal{H}_p \to \mathcal{H}_q$. 
Consider now a horizontal Brownian motion $\xi_t$ started from $p \in M$. For $q \in M$ such that $(p,q) \in \mathcal{C}$ we can construct a process $\tilde{\xi}_t$ that solves the stochastic differential equation
\begin{align*}
\begin{cases}
d\tilde{\xi}_t=\mathfrak{P}_{p \to q} \circ d\xi_t \\
\tilde{\xi}_0=q.
\end{cases}
\end{align*}
This process $\tilde{\xi}$ is well defined up to the random time $\tau_{p,q} =\inf \{ t \ge 0 | (\xi_t,\tilde{\xi}_t) \in \mathcal{C} \}$. For simplicity of the argument, we assume in the following that $\tau_{p,q}=+\infty$ a.s. If not, the technical difficulty can be overcome using the known and now classical techniques to extend  couplings \textit{beyond the cut-locus}, see Wang  \cite{wang-coupling} \& \cite[Theorem 2.3.2]{MR3154951} and Kuwada \cite{Kuwada-coupling}. Note that in such constructions the two processes become identical in the case they meet.

 By construction our process $(\xi_t,\tilde{\xi}_t)$ is then a diffusion process with generator $\Delta^c_{\Ho,\Ho}$ given by \eqref{coupled laplacian}.  When applying It\^o's formula, one sees that the local martingale part of $d(\xi_t,\tilde{\xi}_t)$ is the same as 
\begin{align}\label{quadratic part}
\int_0^t \left\langle \nabla^1 d(\xi_t,\tilde{\xi}_t), \circ d\xi_t \right\rangle +\left\langle \nabla^2 d(\xi_t,\tilde{\xi}_t), \circ d\tilde{\xi}_t \right\rangle
\end{align}
where $\nabla^1$ is the gradient with respect to the first variable and $\nabla^2$ the gradient with respect to the second variable. Since $| \nabla^1 d|=| \nabla^2 d|=1$ and both $\xi$ and $\tilde{\xi}$ are horizontal Brownian motions, the quadratic variation of \eqref{quadratic part} is bounded above by $8t$. Thus \eqref{quadratic part} is a martingale. One deduces
\[
\mathbb{E}(d(\xi_t,\tilde{\xi}_t))\le d(p,q)+\int_0^t \Delta^c_{\Ho,\Ho} d (\xi_s,\tilde{\xi}_s)ds
\]
Thanks to Theorem \ref{comparison coupled laplacian} and trivial estimates on the functions $\tan$ or $\tanh$ we therefore obtain
\[
\mathbb{E}(d(\xi_t,\tilde{\xi}_t))\le d(p,q)-K\int_0^t  d (\xi_s,\tilde{\xi}_s)ds.
\]
This gives
\[
\mathbb{E}(d(\xi_t,\tilde{\xi}_t)) \le  d(p,q) e^{-Kt}.
\]
Let now $f \in \mathbf{Lip}_b (M)$. One has
\begin{align*}
|P_tf(q)-P_tf(p)|&= \left| \mathbb{E} \left( f(\tilde{\xi}_t) \right)- \mathbb{E}\left(f(\xi_t)\right) \right| \\
 &\le \mathbb{E} \left( \left| f(\tilde{\xi}_t)- f(\xi_t)\right| \right) \\
 &\le \mathrm{Lip} ( f) \mathbb{E} \left(d(\xi_t,\tilde{\xi}_t) \right) \\
 &\le \mathrm{Lip} ( f) d(p,q) e^{-Kt}.
\end{align*} 
Since $\mathcal C$ is dense in $M \times M$ and $P_tf$ continuous by hypoellipticity of $\Delta_\Ho$ we conclude that for every $p,q \in M$,
\[
|P_tf(q)-P_tf(p)| \le \mathrm{Lip} ( f)  e^{-Kt}d(p,q).
\]
\end{proof}

\begin{remark}
Estimates of the type $\mathrm{Lip} (P_t f)\le e^{-Kt} \mathrm{Lip} ( f)$  play an important role in the theory of generalized curvature dimension inequalities \cite{BBG14,BG170}.
\end{remark}

\begin{example}
Since any Carnot group satisfies $\mathfrak{R}(U,U) \ge K |U|^2$ for some $K \le 0$, we deduce the  estimate $\mathrm{Lip} (P_t f)\le e^{-Kt} \mathrm{Lip} ( f)$.
\end{example}

Interestingly, in the case of Riemannian foliations with totally geodesic leaves one can obtain stronger gradient bounds for the heat semigroups under slightly different assumptions.

\begin{theorem}\label{GB total}
 Assume that the foliation is totally geodesic and that there exists $K \in \mathbb{R}$ such that for every $U \in \mathfrak{X}(TM)$
\[
\mathfrak{R}(U,U) -\frac{1}{4} (J_U,J_U)_\Ho \ge K |U|^2.
\] 
Then for every $f \in \mathbf{Lip}_b (M) \cap C^2(M)$ and $t \ge 0$, $P_tf \in \mathbf{Lip}_b (M)$ and we have
\[
|\nabla P_t f|(x)\le e^{-Kt} P_t |\nabla f| (x), \quad x \in M.
\] 
\end{theorem}

\begin{proof}
We use a  different coupling. The idea is to consider the following connection
\[
\overset{\circ}{\nabla}_X Y=\nabla_XY +J_XY.
\]
Since the foliation is totally geodesic one can check that $\overset{\circ}{\nabla}$ satisfies the following properties which will be key in the following arguments:
\begin{enumerate}
\item $\overset{\circ}{\nabla}$ is a metric connection;
\item For every $X \in \mathfrak{X}(\Ho)$ and $Y \in \mathfrak{X}(TM)$, $\overset{\circ}{\nabla}_Y X \in \mathfrak{X}(\Ho)$;
\item The geodesic equation can be written $\overset{\circ}{\nabla}_{\dot \gamma} \dot\gamma=0$, see \eqref{geodesic1}.
\end{enumerate}
For $p,q \in \mathcal{C}$ define the transport map $\overset{\circ}{\mathfrak{P}}_{p \to q}:T_pM \to T_qM$ given for $v \in T_pM$ by $\overset{\circ}{\mathfrak{P}}_{p \to q}(v)=\overset{\circ}{v}$ where $\overset{\circ}{v}$ is the $\overset{\circ}{\nabla}$-parallel transport of $v$ along the unit speed geodesic $\gamma:[0,r] \to M$ with endpoints $p,q$. As before, consider  a horizontal Brownian motion $\xi_t$ started from $p \in M$. For $q \in M$ such that $(p,q) \in \mathcal{C}$ we construct the process $\tilde{\xi}_t$ that solves the stochastic differential equation
\begin{align*}
\begin{cases}
d\tilde{\xi}_t=\overset{\circ}{\mathfrak{P}}_{p \to q} \circ d\xi_t \\
\tilde{\xi}_0=q.
\end{cases}
\end{align*}
As before, we will  ignore the genuine but inessential difficulties related to cut-locus and assume that the couple $(\xi_t,\tilde{\xi}_t)$ is defined for all $t$'s. 

By construction, the process is a diffusion on $M \times M$ with generator
\begin{align*}
\overset{\circ}{\Delta}\,^c_{\Ho,\Ho} f (p,q)=\Delta^1_{\Ho} f (p,q)+\Delta^2_{\Ho} f (p,q) +2 \sum_{i=1}^n v_i^1 \overset{\circ}{v}\,^2_i f(p,q).
\end{align*}
In It\^o's formula the local martingale part of $d(\xi_t,\tilde{\xi}_t)$ is the same as the local martingale part of
\begin{align}\label{quadratic part }
\int_0^t \left\langle \nabla^1 d(\xi_t,\tilde{\xi}_t), \circ d\xi_t \right\rangle +\left\langle \nabla^2 d(\xi_t,\tilde{\xi}_t), \circ d\tilde{\xi}_t \right\rangle
\end{align}
which vanishes since $\nabla^1 d(\xi_t,\tilde{\xi}_t) +\overset{\circ}{\mathfrak{P}}_{\tilde{\xi}_t \to \xi_t}\nabla^2 d(\xi_t,\tilde{\xi}_t)=0$ from the fact that any geodesic satisfies $\overset{\circ}{\nabla}_{\dot \gamma} \dot\gamma=0$.

We now need an estimate on $\overset{\circ}{\Delta}\,^c_{\Ho,\Ho} d (p,q)$. The argument is essentially the same as in the proof of Theorem \ref{comparison laplacian}. From the index lemma one has
\[
\overset{\circ}{\Delta}\,^c_{\Ho,\Ho} d (p,q) \le  \sum_{i=1}^n I(\gamma,Y_i)
\]
where the $Y_i$'s are arbitrary vector fields along $\gamma$ such that
\[
Y_i(0)=v_i, \, \, Y_i(r)=\overset{\circ}{v}_i.
\]
Pick
\[
Y_i(t)=\mathfrak{c}_K (t) X_i(t),
\]
where $X_i$ is the $\overset{\circ}{\nabla}$-parallel transport of $v_i$ along $\gamma$ and $\mathfrak{c}_K$ is as in the proof of Theorem \ref{comparison laplacian}. We then compute
\begin{align*}
 \sum_{i=1}^n I(\gamma,Y_i)  &=  \sum_{i=1}^n  \int_0^r  \left | \nabla_{\dot \gamma}  Y_i +\frac{1}{2} J_{\dot \gamma}  Y_i   \right|^2  -\left \langle \mathrm{Riem}^\nabla(\dot \gamma,Y_i)Y_i, \dot \gamma \right \rangle  +\left\langle (\nabla_{Y_i} \Tor^\nabla)(Y_i,\dot\gamma),\dot \gamma_\V \right\rangle\\
 &\quad -\frac{1}{4} | J_{\dot \gamma} Y_i |^2_\Ho+ |\Tor^\nabla(Y_i,\dot\gamma)|^2 \, dt \\
 &=\sum_{i=1}^n  \int_0^r  \mathfrak{c}'_K(t)^2  - \mathfrak{c}_K(t)^2 \left(  \left \langle \mathrm{Riem}^\nabla(\dot \gamma,X_i)X_i, \dot \gamma \right \rangle  +\left\langle (\nabla_{X_i} \Tor^\nabla)(X_i,\dot\gamma),\dot \gamma_\V \right\rangle \right. \\
 &\left. \quad + |\Tor^\nabla(X_i,\dot\gamma)|^2 \right) \, dt \\
 &=  \int_0^r n  \mathfrak{c}'_K(t)^2-\mathfrak{c}_K(t)^2\left( \mathfrak{R}(\dot \gamma,\dot \gamma) +\frac{1}{4} (J_{\dot \gamma},J_{\dot \gamma})_\Ho \right)dt \\
 &\le  \int_0^r n  \mathfrak{c}'_K(t)^2- K \mathfrak{c}_K(t)^2  dt \\
 &\le -Kr.
\end{align*}
Therefore, we have
\[
d(\xi_t,\tilde{\xi}_t) \le d(p,q)-K \int_0^t d(\xi_s,\tilde{\xi}_s) ds
\]
which yields
\[
d(\xi_t,\tilde{\xi}_t) \le  d(p,q) e^{-Kt}.
\]
Let now $f \in \mathbf{Lip}_b (M)\cap C^2(M)$. One has
\begin{align*}
|P_tf(q)-P_tf(p)|&= \left| \mathbb{E} \left( f(\tilde{\xi}_t) \right)- \mathbb{E}\left(f(\xi_t)\right) \right| \\
 &\le \mathbb{E} \left( \left| f(\tilde{\xi}_t)- f(\xi_t)\right| \right) \\
 &\le   \mathbb{E} \left(d(\xi_t,\tilde{\xi}_t)  \sup_{\rho \in B(\xi_t,d(p,q) e^{-Kt})} \left| \nabla f (\rho)\right| \right) \\
 &\le d(p,q) e^{-Kt}\mathbb{E} \left(  \sup_{\rho \in B(\xi_t,d(p,q) e^{-Kt})} \left| \nabla f (\rho)\right| \right).
\end{align*} 
Using the density of $\mathcal{C}$ in $M\times M$ and letting $q \to p$ yields
\[
|\nabla P_t f|(p)\le e^{-Kt} P_t |\nabla f| (p).
\]
\end{proof}

As an example of application of the previous result, consider the case of Carnot groups.  Consider a Carnot group $G$ of step 2 equipped with a foliation $(\mathcal{F},g)$ like in Proposition \ref{Example Carnot}. This foliation is totally geodesic. 
Consider then the canonical variation of the metric $g$ given by
\[
g_\varepsilon=g_\Ho \oplus \frac{1}{\varepsilon} g_\V, \quad \varepsilon >0.
\]
One can easily check that the foliation $(\mathcal{F},g_\varepsilon)$ is Riemannian and totally geodesic for every $\varepsilon >0$. Denote $\mathfrak{R}_\varepsilon$ the tensor $\mathfrak{R}$ of $(\mathcal{F},g_\varepsilon)$. It follows from Proposition  \ref{example carnot r} that for $U,V \in \mathfrak{X}(TM)$
\begin{align*}
\mathfrak{R}_\varepsilon(U,V)= \sum_{i=1}^n \frac{1}{4\varepsilon^2} \langle J_{U} X_i,J_V X_i\rangle_\Ho-\frac{1}{\varepsilon} \langle \Tor^\nabla(X_i,U),\Tor^\nabla(X_i,V)\rangle_\V.
\end{align*}
Therefore there exists a constant $K <0$ such that for $U\in \mathfrak{X}(TM)$
\begin{align*}
\mathfrak{R}_\varepsilon(U)- \sum_{i=1}^n \frac{1}{4\varepsilon^2} \langle J_{U} X_i,J_V X_i\rangle_\Ho \ge-\frac{K}{\varepsilon} |U|^2_\varepsilon.
\end{align*}
Applying Theorem \ref{GB total} we recover a gradient bound first observed in \cite{MR3601645}:
\[
\sqrt{ |\nabla_\Ho P_t f|^2 +\varepsilon |\nabla_\V P_t f|^2} \le e^{Kt/\varepsilon} P_t \left( \sqrt{ |\nabla_\Ho P_t f|^2 +\varepsilon |\nabla_\V P_t f|^2 } \right).
\]
Note that choosing $\varepsilon=t$ yields the following result:
\begin{corollary}
Let $G$ be a Carnot group of step 2 and $P_t$ be the sub-Riemannian semigroup. There exists a constant $C>1$ such that for every $f \in \mathbf{Lip}_b (M) \cap C^2(M)$
\[
 |\nabla_\Ho P_t f| +t  |\nabla_\V P_t f| \le C \left[P_t  \left(  |\nabla_\Ho f|\right) +t  P_t  \left( |\nabla_\V f| \right)\right]
\]
\end{corollary}

\bibliographystyle{plain}
\bibliography{biblio}
\end{document}